\documentclass[10pt,
twoside]{amsart}
\usepackage{lmodern} 
\usepackage{soul, wrapfig}
\usepackage{graphicx}
\usepackage{color}
\usepackage{amsmath}
\usepackage{amssymb}
\usepackage{amsfonts}
\usepackage{latexsym, amsthm, mathrsfs}
\usepackage{hyperref, breakurl}
\usepackage[OT1]{fontenc}
\usepackage{fancybox}
\usepackage{amscd}
\usepackage{enumitem}
\usepackage{fontenc}
\usepackage{amsthm}
\usepackage{graphicx}
\usepackage[english]{babel}
\usepackage{tikz-qtree}
\usepackage{rotating}
\usepackage{stmaryrd}

\usetikzlibrary{arrows, automata, positioning}
\newtheorem{definition2}{Definition}

\newtheorem{lemma}[definition2]{Lemma}
\newtheorem{claim}[definition2]{Claim}
\newtheorem{sublemma}{Sublemma}
\newtheorem{remark2}[definition2]{Remark}
\newtheorem{theorem}[definition2]{Theorem}

\newtheorem*{Mobservation2}{Main Observation}
\newtheorem{example2}[definition2]{Example}
\newtheorem{corollary}[definition2]{Corollary}
\newtheorem{proposition}[definition2]{Proposition}

\newenvironment{definition}{\begin{definition2} \upshape}{\end{definition2}}
\newenvironment{remark}{\begin{remark2} \upshape}{\end{remark2}}

\newcommand{\algebra}{}
\newcommand{\algebraQ}{\algebra{B}}
\newcommand{\amal}{\textsf{Am}}

\newcommand{\baire}{\omega^\omega}

\newcommand{\cantor}{2^\omega}

\newcommand{\cohen}{\poset{C}}
\newcommand{\conc}{\smallfrown}

\newcommand{\DDelta}{\mathbf{\Delta}}

\newcommand{\dom}{\text{dom}}

\newcommand{\enfa}{\textit}

\newcommand{\ideal}{\mathcal}

\newcommand{\ifif}{\Leftrightarrow}
\newcommand{\force}{\Vdash}

\newcommand{\mathias}{\poset{M}}

\newcommand{\On}{\text{On}}

\newcommand{\poset}{\mathbb}

\newcommand{\random}{\poset{B}}

\newcommand{\real}{\mathbb{R}}

\newcommand{\restric}{{\upharpoonright}}

\newcommand{\SSigma}{\mathbf{\Sigma}}

\newcommand{\stem}{\textsf{stem}}

\newcommand{\supp}{\textsf{supp}}

\newcommand{\term}{\textsc{Term}}

\newcommand{\nl}{\mathbf{NL}}

\newcommand{\nr}{\mathbf{NR}}

\newcommand{\U}{\textbf{U}}
\newcommand{\SP}{\textbf{SPA}}
\newcommand{\IP}{\textbf{IPA}}
\newcommand{\WP}{\textbf{WPA}}
\newcommand{\NB}{\textbf{NB}}
\newcommand{\NL}{\textbf{NL}}
\newcommand{\NR}{\textbf{NR}}

\newcommand{\w}{\omega}

\begin{document}
\title[Social welfare relations and irregular sets
]{Social welfare relations and irregular sets}
\author{Ram Sewak Dubey}
\address[R. S. Dubey]{Department of Economics, Feliciano School of Business, Montclair State University, Montclair, NJ, 07043, USA.}
\email{dubeyr@montclair.edu}

\author{Giorgio Laguzzi}
\address[G. Laguzzi]{Albert-Ludwigs-Universit\"at Freiburg, Mathematische Institut, Abteilung f\"ur Mathematische Logik,  Ernst-Zermelo str. 1, 79104 Freiburg im Breisgau, Germany}
\email[Corresponding author]{giorgio.laguzzi@libero.it}

\maketitle

\begin{abstract}
Social welfare relations satisfying Pareto and equity principles on infinite utility streams has revealed a non-constructive nature. In this paper we study more deeply the needed fragments of AC and their relations with other well-known non-constructive sets. We also prove some connections with the Baire property, answering Problem 11.14 posed in \cite{Mathias}.
\end{abstract}

\section{Introduction and basic definitions}

In recent years, various papers have shown some interplay between theoretical economics and mathematical logic. More specifically, some connections have risen between social welfare relations on infinite utility streams and descriptive set theory. In particular the following results have been proven:
\begin{itemize}
\item in \cite{Lau09} Lauwers proves that the existence of a total social welfare relation satisfying infinite Pareto and anonymity implies the existence of a non-Ramsey set.
\item in \cite{Z07} Zame proves that the existence of a total social welfare relation satisfying strong Pareto and anonymity implies the existence of a non-Lebesgue measurable set.
\end{itemize}
(Precise definitions of these combinatorial concepts from economic theory are introduced in Definition \ref{def1} below.)

So in terms of set-theoretical considerations, these results mean that the existence of these specific relations satisfying certain combinatorial principles from economic theory are connected to a fragment of the axiom of choice, AC. 
As a consequence, from the set-theoretical point of view, it is natural and interesting to understand more deeply the \emph{exact} fragment of AC they correspond to, in particular compared to
other objects coming from measure theory, topology and infinitary combinatorics, extensively studied in the set-theoretic literature (for a detailed overview see \cite{Ikegami},\cite{Khomskii},\cite{BL99}). More precisely, we show that the reverse implications of Lauwers and Zame's results do not hold, and therefore total social welfare relations satisfying Pareto and anonymity need a strictly larger fragment of AC than non-Lebesgue measurable and non-Ramsey sets. 
Moreover we are going to analyse social welfare relations from the topological point of view, and specifically the connection with the Baire property.
This question is explicitely asked in \cite[Problem 11.14]{Mathias} and it was the main motivation arousing this paper. We deeply thank Adrian Mathias for such a fruitful inspiration. 

\vspace{3mm}

Since the motivation of this paper comes from the study of some combinatorial concepts studied in economic theory, we briefly remind the basic notions about social welfare relations and infinite utility streams, in as much detail as required for our scope.

We consider a \emph{set of utility levels} $Y$ (or \emph{utility domain}) endowed with some topology and totally ordered, and we call $X:= Y^\omega$ the corresponding \emph{space of infinite utility streams}, endowed with the product topology.
Given $x,y \in X$ we write $x \leq y$ iff $\forall n \in \omega (x(n) \leq y(n))$, and $x < y$ iff $x \leq y \land \exists n \in \omega (x(n) < y(n))$. Furthermore we set $\ideal{F}:= \{ \pi: \omega \rightarrow \omega: \text{ finite permutation} \}$, and we define, for $x \in X$, $f_\pi(x):= \langle x({\pi(n)}): n \in \omega \rangle$.

We say that $\precsim$ subset of $X \times X$ is a \emph{social welfare relation (SWR) on $X$} iff $\precsim$ is reflexive and transitive.
Next we introduce the theoretical economic principles used in this paper.

\begin{definition} \label{def1} 
Let $\precsim$ be a SWR on $X$. We say that $\precsim$ satisfies:
\begin{itemize}
\item \label{D1} \emph{Anonymity (A)} iff whenever given $x, y\in X$ there exist $i, j\in \w$ such that $y(j) = x(i)$ and $x(j) = y(i)$, while $y(k) = x(k)$ for all $k\in \w\setminus \{i, j\}$, then $x \sim y$.
\item \label{D2} \emph{Strong Pareto (SP)} iff for all $x, y\in X$, if $x\leq y$ and $x(i)<y(i)$ for some $i \in \w$, then $x \prec y$.
\item \label{D3} \emph{Infinite Pareto (IP)} iff for all $x, y\in X$, if $x\leq y$ and $x(i)<y(i)$ for infinitely many  $i \in \w$, then $x \prec y$.
\item \label{D4} \emph{Weak Pareto (WP)} iff for all $x, y\in X$, if $x(i)<y(i)$ for all $i \in \w$, then $x \prec y$.
\end{itemize}
\end{definition}
It is immediate to notice that SP $\Rightarrow$ IP $\Rightarrow$ WP. 

From descriptive set theory of the reals we recall the following notions:
\begin{itemize}
\item $X \subseteq [\omega]^\omega$ is \emph{non-Ramsey} iff for every $F \in [\omega]^\omega$ one has $[F]^\omega \not \subseteq X$ and $[F]^\omega \cap X \neq \emptyset$.
\item $X 	\subseteq 2^\omega$ is Lebesgue measurable iff there exists a Borel set $B \subseteq 2^\omega$ such that $X \Delta B$ has measure zero. In case a set is not Lebesgue measurable we call it \emph{non-Lebesgue}.
\item  $X \subseteq 2^\omega$ satisfies the \emph{Baire property} iff there exists an open set $O \subseteq 2^\omega$ such that $X \Delta O$ is meager. In case a set does not satisfy the Baire property  we call it \emph{non-Baire}.
\end{itemize}

Throughout the paper, we use the following notation:
\[
\begin{split}
\SP_Y :=& \text{``There exists a total SWR on $Y^\omega$ satisfying A and SP"} \\
\IP_Y :=& \text{``There exists a total SWR on $Y^\omega$ satisfying A and IP"} \\
\WP_Y :=& \text{``There exists a total SWR on $Y^\omega$ satisfying A and WP"} \\
\NL :=& \text{``There exists a non-Lebesgue set"}\\
\NR :=& \text{``There exists a non-Ramsey set"}\\
\NB :=& \text{``There exists a non-Baire set"}
\end{split}
\]

\begin{remark} \label{Remark2}
In the first three symbols, involving statements on total SWRs, we have specified the utility domain, as the nature of such SWRs strongly depends on $Y$. To see this, one can observe that combining A and WP is trivial when $Y=\{ 0,1 \}$, since for instance we can simply define $\prec$ so that for all $y \in 2^\omega$ which are not the constant sequence $e_0:= \langle 0, 0, \dots \rangle$ we have $e_0 \prec y$, but on the other hand the combination of A and WP gives non-constructive SWRs when $Y=[0,1]$ (for instance, see Proposition \ref{P-wp}).
\end{remark}

\begin{remark}
The study on infinite populations and these combinatorial principles has been rather extensively developed in the economic literature. Summarizing the reasons and analysing the intepretations in the context of economic theory is away from the aim of this paper, which should be meant as a contribution to a set-theoretic question coming from the study of combinatorial concepts introduced in economic theory, rather than an effective application of set theory to economic theory. The reader interested in detailed background from economic theory literature may consult the following selected list of papers: \cite{Asheim2010}, \cite{Au64}, \cite{Chi*}, \cite{Chi96}, \cite{Litak}, \cite{Lau09}, \cite{LV97}, \cite{Z07}.
\end{remark}

\vspace{3mm}

The technical tools from forcing and descriptive set theory of the reals are introduced through the paper when specifically needed.

\section{Topological side of SWRs} \label{section1}
In this section we investigate the topological properties of SWRs satisfying Pareto and anonymity (focusing on the Baire/product topology), and proving some interplay with non-Baire sets, answering Problem 11.14 posed in \cite{Mathias}.

We start with an example. Let $X := [0,1]^\omega$ and
 define $\rhd$ (usually called Suppes-Sen principle) as follows: for every $x,y \in X$, we say 
\[
\begin{split}
x \rhd y & \text{ iff there exists $\pi \in \mathcal{F}$ such that $f_\pi(x) > y$}.\\
x \sim y & \text{ iff there exists $\pi \in \mathcal{F}$ such that $f_\pi(x) = y$}.
\end{split}
\]
Let $\supp_r(y):= \{ n \in \omega: y(n)\neq r \}$, for a given $r \in [0,1]$.
It is clear that $\rhd$ is a SWR satisfying SP and A.
We consider the standard euclidean topology on $[0,1]$ and then the corresponding product topology on $X:=[0,1]^\omega$.

\begin{remark}
The Suppes-Sen principle is rather coarse from the topological point of view, as many pairs $x,y \in X$ are incompatible w.r.t. $\rhd$. 
More precisely, $S:= \{(x,y) \in X \times X: x  \ntriangleright y \land y  \ntriangleright x \land x \not \sim y \}$ is comeager. 

Let $S'$ be the complement of $S$. 
We show that $S'$ is meager. 
First partition $S'$ into three pieces: 
$E:= \{ (x,y) \in X \times X: x \rhd y \}$, $D:= \{ (x,y) \in X \times X: y \rhd  x \}$ and $C:= \{ (x,y) \in X \times X: y \sim  x\}$.
We check that $E$ is meager and then note that similar arguments work for $D$ and $C$ as well.
Fix $y \in X$ so that $\supp_0(y)$ is infinite (i.e., $y$ is not eventually 0) and consider $E^y:= \{ x \in X : (x,y) \in E \}$. Let $H^y:= \{x \in X:  x > y \}$.
Note that 
\(
E^y:= \bigcup_{\pi \in \mathcal{F}} H^{f_\pi(y)}.
\)
Since $\mathcal{F}$ is countable it is enough to prove that for each $\pi \in \mathcal{F}$, $H^{f_\pi(y)}$ is meager.
Actually $H^y$ is nowhere dense, for every $y \in X$ with $|\supp_0(y)| =\omega$; in fact, given $U:= \prod_{n \in \omega} U_n \subseteq X$ basic open set and $k \in \omega$ sufficiently large that for all $n \geq k$, $U_n=[0,1]$,
one can pick $n^* > k$ such that $n^* \in\supp_0(y)$ and pick $U' \subseteq U$ so that: 
$\forall n \neq n^*$, $U_n=U'_n$ and
$U'_{n^*}:= [0,y(n^*))$.
Then it is clear that $U' \cap H^y=\emptyset$. Note that if $\pi \in \mathcal{F}$ we get $|\supp_0(f_\pi(y))|=\omega$ as well, and so $H^{f_\pi(y)}$ is nowhere dense too.
By Ulam-Kuratowski theorem, the proof is concluded simply by noticing that the set $\{y \in X: |\supp_0(y)|=\omega  \}$ is comeager, which easily follows since each $B_n:= \{ y \in B: |\supp_0(y)| \leq n \}$ in nowhere dense. 

\end{remark}

Under this point of view the Suppes-Sen principle can then be considered rather ``poor", for a positive characteristic of a SWR is to be able of comparing as many elements as possible. 
Part 1 of the following proposition shows that this coarse nature of such SWRs is not only specific for Suppes-Sen principle, but in a sense it holds for any ``regular" SWR satisfying A and SP. Moreover, in part 2 we show that when assuming $\SP_Y$, the price to pay is to get a set without Baire property. In the following we consider $X=[0,1]^\omega$.

\begin{proposition} \label{prop:non-BP}
Let $X=[0,1]^\omega$. Then the following hold:
\begin{enumerate}
\item Let $\precsim$ be a SWR satisfying A and SP on $X$, $E:= \{(x,y) \in X \times X: x  \succ y\}$ and $D:=\{(x,y) \in X \times X: y  \succ x  \}$.
If both $E$ and $D$ have the Baire property, then $E \cup D$ is meager.
\item Let $\precsim$ be a total SWR satisfying A and SP on $X$, and let E,D as above.
Then either $E$ or $D$ does not have the Baire property.
\end{enumerate}
\end{proposition}

\begin{proof}
Under the assumption $E,D$ both satisfying the Baire property, we show that $E$ is meager, and remark that the argument for $D$ is essentially the same. 
Given $S \subseteq X \times X$ and $y \in X$, we use the notation $S_y:= \{x \in X: (x,y) \in S  \}$.
Since we assume $E$ has the Baire property, we can find a Borel set $B \subseteq E$ such that $E \setminus B$ is meager; moreover for every $\pi, \pi' \in \mathcal{F}$ we can define $B(\pi,\pi'):= \{ (f_\pi(x), f_{\pi'}(y)): (x,y) \in B \}$. Put $B^*:= \bigcup \{ B(\pi,\pi'): \pi, \pi' \in \mathcal{F}  \}$ and note that $B^* \subseteq E$, as $E$ is closed under finite permutations. Moreover, $E \setminus B^*$ is meager too. 
Let $I_0:= \{ y \in X : B^*_y  \text{ is meager} \}$ and $I_1:= \{ y \in X : B^*_y  \text{ is comeager} \}$. Note that each $B^*_y$ is by definition invariant under finite permutations,
i.e., $x \in B^*_y \ifif f_\pi(x) \in B^*_y$, where $\pi \in \mathcal{F}$. 

 Hence by \cite[Theorem 8.46]{Kechris} with $G$ being the group on $X$ induced by finite permutations, we have that each $B^*_y$ is either meager or comeager, and hence $I_0 \cup I_1 = X$. We also observe that both $I_0$ and $I_1$ are invariant under $\pi \in \mathcal{F}$. In fact, it is straightforward to check that if $\pi \in \mathcal{F}$ and $B^*_y$ is meager,  then $B^*_{f_\pi(y)}$ is meager too. 
So, if $I_1$ is comeager, by Kuratowski-Ulam theorem we get $E$ is comeager. But since an analogous argument could be done for $D$ too, we would have that also $D$ is comeager; however by definition $E \cap D = \emptyset$, which is a contradiction.
As a consequence, we get $I_0$ is comeager, which implies $E$ (and $D$ as well) is meager.  

\vspace{2mm}

$2.$ Note that in this case the SWR is total and so, if $E$ and $D$ both satisfy the Baire property, it follows that the set $A:= \{ (x,y) \in X \times X: x \sim y \}$ is comeager. 
Thus, by Kuratowski-Ulam's theorem there is $y \in X$ such that $A_{y}$ is comeager. 
Pick $0<r<\frac{1}{2}$, define 
\[
H:= [0,1-r] \times \prod_{i \in \omega} [0,1],
\] 
and consider the injective function $\phi: X' \rightarrow X$ such that $i(x(0)):= x(0)+ r$. Note that 
\[
\phi[H]:= [r,1] \times \prod_{i \in \omega} \big [0,1],
\]
Note also that for every $x \in H$, $\phi(x) \succ x$ by Pareto, and so in particular $x \sim y \ifif x \not \sim \phi(y)$. Hence, we have the following two mutually contradictory consequences.
\begin{itemize}
\item On the one side, $H \cap A_y \cap \phi[H \cap A_y]=\emptyset$;
indeed if there exists $z \in H \cap A_y \cap \phi[H \cap A_y]$, then there is $x \in H \cap A_y$ such that $z := \phi(x)$; then on the one hand we have $z \in A_y$ which gives $z \sim y$, but on the other hand we have $x \in H \cap A_y$ that in turn gives $x \sim y$ and so together with $x \prec \phi(x)=z$ we would get $y \prec z$; contradiction. 
\item On the other side, $H \cap A_y \cap \phi[H \cap A_y]$ cannot be meager, since $H \cap \phi[H]$ is a non-empty open set, $H \cap A_y$ is comeager in $H$ and $\phi[H \cap A_y]$ is comeager in $\phi[H]$.
\end{itemize}
\end{proof}

\begin{remark}
Note that Proposition \ref{prop:non-BP} holds even when $Y=\{0,1\}$. For part (1) we can argue with the same proof, while for part (2) we only need to consider the map $\phi: X \rightarrow X$ such that $\phi(x)(0)\neq x(0)$ and for all $n > 0$, $\phi(x)(n)=x(n)$, and so $\phi(x) \not \sim x$.  
\end{remark}

\subsection{Mathias-Silver trees}
We recall the standard basic notions and notation about tree-forcings.
A subset $T \subseteq Y^{<\w}$ is called a \emph{tree} if and only if for every $t \in T$ every $s \subseteq t$ is in $T$ too, in other words, $T$ is \emph{closed} under initial segments.
We call the segments $t \in T$ the \emph{nodes} of $T$ and denote the \emph{length} of the node by $|t|$; $\stem(T)$ is the longest node such that $\forall t \in T (t \subseteq \stem(T) \vee \stem(T) \subseteq t)$. 
A node $t \in T$ is called \emph{splitting} if there are two distinct $n$, $m \in Y$ such that $t^\conc n$, $t^\conc m \in T$.
Given $x \in Y^{\w}$  and $n \in \w$, we denote by $x \restric n$ the cut of $x$ of length $n$, i.e., $x \restric n := \langle x(0), x(1), \cdots, x(n-1) \rangle$.
A tree $p \subseteq Y^{<\w}$ is called \emph{perfect} if and only if for every $s \in p$ there exists $t \supseteq s$ splitting.
We define $[p]:=\{x \in Y^\omega: \forall n \in \w (x\restric n \in p)\}$, and $x \in [p]$ is called a \emph{branch of $p$}. 

A tree $p \subseteq 2^{<\w}$ is called \emph{Silver} tree if and only if $p$ is perfect and for every $s$, $t \in p$, with $|s|=|t|$ one has $s^\conc 0 \in p$ $\ifif t^\conc 0 \in p$ and $s^\conc 1 \in p$  $\ifif t^\conc 1 \in p$.
If $t$ is a splitting node of $p$, we call $|t|+1$ a splitting level of $p$ and let $S(p)$ denote the set of splitting levels of $p$. 
Then set $U(p):= \{n \in \w: \forall x \in [p] (x(n)=1)\}$ and let $\{n^p_k: k \in \w\}$ enumerate the set $S(p) \cup U(p)$. 

We could also define a Silver tree $p$ and its corresponding set of branches $[p]$ relying on the notion of partial functions. 
Consider a partial function $f: \w \rightarrow \{0, 1\}$ such that $\dom(f)$ is co-infinite (i.e. the complement of the domain of $f$ is infinite); then define $N_f:= \{x \in 2^{\w}: \forall n \in \dom(f) (f(n)=x(n))\}$. 
It easily follows from the definitions that there is a one-to-one correspondence between every Silver tree $p$ and a set $N_f$;
given any Silver tree $p$ there is a unique partial function $f: \w \rightarrow \{0, 1\}$ such that $[p]=N_f$.
In particular, the set of splitting levels $S(p)$ correspond to $\w \setminus \dom(f)$. 
Silver trees are extensively studied in the literature, as well as their topological properties (e.g., see \cite{Halbeisen2003} and \cite{BLH2005})

We now introduce a variant of Silver trees which perfectly serves for our purpose.

\begin{definition}
Let $p \subseteq 2^{<\w}$ be a Silver tree with $\{ n_k^p:k \geq 1 \}$ enumeration of $S(p) \cup U(p)$; $p$ is called a \emph{Mathias-Silver} tree ($p \in \poset{MV}$) if and only if there are infinitely many triples $(n^p_{m_j}, n^p_{m_j+1}, n^p_{m_j+2})$'s such that:
\begin{enumerate}
\item for all $j \geq 1$, $m_j$ is even;
\item for all $j \geq 1$,  $n^p_{m_j}, n^p_{m_j+1}, n^p_{m_j+2}$ are in $S(p)$ with $n^p_{m_j} +1< n^p_{m_j+1}$ and $n^p_{m_j+1}+1 < n^p_{m_j+2}$;
\item for all $j \geq 1$, $t \in p$, $i < |t|$ $(n^p_{m_j} < i < n^p_{m_j+1} \vee n^p_{m_j+1} < i < n^p_{m_j+2}  \Rightarrow t(i)=0)$.
\end{enumerate}
We call $(n^p_{m_j}, n^p_{m_j+1}, n^p_{m_j+2})$ satisfying (1), (2) and (3) a \emph{Mathias triple}. 
\end{definition}
\begin{remark}
The idea is that a Mathias-Silver tree is a special type of a Silver tree that mimics infinitely often the feature of a Mathias tree, which is that in between the splitting levels occuring in any Mathias triple $(n^p_{m_j}, n^p_{m_j+1}, n^p_{m_j+2})$ all nodes of the tree $p$ take value 0. In the proof of propositions \ref{P-ip} and \ref{P-wp} this property will be crucial, and indeed it is not clear how to obtain, if possible, similar results working with Silver trees instead of Mathias-Silver trees.
\end{remark}

\begin{definition}
A set $X \subseteq 2^\w$ is called \emph{Mathias-Silver measurable set} (or $\poset{MV}$-measurable set) if and only if there exists $p \in \poset{MV}$ such that $[p] \subseteq X$ or $[p] \cap X = \emptyset$.
A set $X \subseteq 2^\w$ not satisfying this condition is called a \emph{non-$\poset{MV}$-measurable set}.
\end{definition}

\noindent The following lemma is the key step to prove that any set satisfying the Baire property is $\poset{MV}$-measurable, or in other words, that a non-$\poset{MV}$-measurable set is a particular instance of a non-Baire set. 
The proof is a variant of the construction developed in \cite{Halbeisen2003} for standard Silver trees. 

\begin{lemma}\label{l1}
Given any comeager set $C \subseteq 2^\w$ there exists $p \in \poset{MV}$ such that $[p] \subseteq C$.
\end{lemma}

\begin{proof} 
Let $\{D_n: n \in \w\}$ be a $\subseteq$-decreasing sequence of open dense sets such that $\underset{n\in\w}{\bigcap} D_n \subseteq C$. 
Given $s \in 2^{<\omega}$, put $N_s:=\{x \in 2^\omega: x \supset s\}$.
Recall that if $D$ is open dense, then $\forall s \in 2^{<\w}$ there exists $s^{\prime} \supseteq s$ such that $N_{s^{\prime}} \subseteq D$.
We construct $p \in \poset{MV}$ by recursively building up its nodes as follows: first of all let 
\[
\begin{split}
s_1= (10000), s_2=(10001), s_3=(10100),  s_4=(10101),\\
s_5=(00000), s_6=(00001), s_7=(00100),  s_8=(00101).
\end{split}
\]
\begin{itemize} 
\item Pick $t_{\emptyset} \in 2^{<\w}$ such that $N_{t_\emptyset} \subseteq D_0$, and then let $F_0:= \overset{8}{\underset{k=1}{\bigcup}}\left\{t_\emptyset^\conc s_k\right\}$ and $T_0$ be the downward closure of $F_0$, i.e., $T_0:= \left\{s \in 2^{<\w}: \exists t \in F_0 (s \subseteq t) \right\}$;
\item Assume $F_n$ is already defined. 
Let $\left\{t_j: j \leq J \right\}$ enumerate all nodes in $F_n$ (note by construction $J=8^{n+1}$). We proceed inductively as follows: pick $r_0 \in 2^{<\w}$ such that $N_{t_0^\conc r_0} \subseteq D_{n+1}$; then pick $r_1 \supseteq r_0$ such that $N_{t_1^\conc r_1} \subseteq D_{n+1}$; proceed inductively in this way for every $j \leq J$, so $r_j \supseteq r_{j-1}$ such that $N_{t_j^\conc r_j} \subseteq D_{n+1}$. 
Finally put $r=r_J$. 
Then define 
\[
\begin{split}
F_{n+1}&:= \bigcup \left\{ t^\conc r^\conc s_k: t \in F_n, k=1,2, \dots 8 \right\} \\
T_{n+1}&:= \left\{ s \in 2^{<\w}: \exists t \in F_{n+1} (s \subseteq t)\right\}.
\end{split}
\]
\end{itemize} 
Note that by construction, for all $t \in F_{n+1}$ we have $N_t \subseteq D_{n+1}$.
Finally put $p := \underset{n\in\w}{\bigcup} T_n$. 
Then by construction $p \in \poset{MV}$ as it is a Silver tree and the use of $s_1, s_2, \cdots, s_8$ ensures that $p$ contains infinitely many Mathias triples, and so $p \in \poset{MV}$. 
It is left to show $[p] \subseteq  \underset{n\in\w}{\bigcap} D_n$. 
For this, fix arbitrarily $x \in [p]$ and $n \in \w$; 
by construction there is $t \in F_n$ such that $t \subset x$.
Since $N_t \subseteq D_n$ we then get $x \in N_t \subseteq D_n$. 
\end{proof}

\begin{corollary} \label{C1}
If $A \subseteq 2^{\w}$ satisfies the Baire property, then $A$ is a $\poset{MV}$-measurable set.
\end{corollary}
\begin{proof}
The proof is a simple application of Lemma \ref{l1} and the fact that any set satisfying the Baire property is either meager or comeager relative to some basic open set $N_t$. 
Indeed, if $A$ is meager, then we apply Lemma \ref{l1} to the complement of $A$ and find $p \in \poset{MV}$ such that $[p] \cap A= \emptyset$.
If there exists $t \in 2^{<\w}$ such that $A$ is comeager in $N_t$, then we can use the construction as in Lemma \ref{l1} in order to find $p \in \poset{MV}$ such that $[p] \subseteq A$, simply by choosing $t_{\emptyset} \supseteq t$, $t_{\emptyset} \in D_0$ and then use the same construction as in the proof of Lemma \ref{l1}.
\end{proof}

\subsection{Infinite Pareto and Anonymity}\label{s4.1}
Given $x \in 2^\w$, let $U(x):= \{ n \in \w: x(n)=1 \}$ and $\{ n^x_k: k \geq 1 \}$ enumerate the numbers in $U(x)$.  
Define
\begin{equation}
\begin{split}
o(x):= [n^x_1,n^x_2) \cup [n^x_3,n^x_4) \cdots [n^x_{2j+1}, n^x_{2j+2}) \cup \cdots \\
e(x):= [n^x_2,n^x_3) \cup [n^x_4,n^x_5) \cdots [n^x_{2j+2}, n^x_{2j+3}) \cup \cdots
\end{split}
\end{equation}
As usual we identify subsets of $\w$ with their characteristic functions, so that we can write $o(x)$, $e(x) \in 2^\w$.

\begin{proposition}\label{P-ip} 
Let $\precsim$ denote a total SWR satisfying IP and A on $X= 2^{\w}$.
Then there exists a subset of $X$ which is not $\poset{MV}$-measurable.
\end{proposition}

Note that Proposition \ref{P-ip} somehow improves the result in part(b) of Proposition \ref{prop:non-BP}, as SP is stronger than IP and by Corollary \ref{C1} any non-$\poset{MV}$-measurable set is non-Baire too. However, Proposition  \ref{prop:non-BP} is still relevant as it reveals some topological structural properties that cannot be deduced from Proposition \ref{P-ip}; for instance, part(a) of Proposition \ref{prop:non-BP} shows that any SWR satisfying the Baire property must have comeager many incompatible or equivalent pairs, which essentially means that any \emph{regular} SWR is necessarily rather coarse as it has \emph{many} either incomparable or indistinguishable pairs, and the non-Baire set built in part(b) shows how the irregularity of total SWRs is intrinsically connected to the characteristic of excluding the presence of too many incompatible elements. 

\begin{proof}[Proof of Proposition \ref{P-ip}]
Let $\Gamma:= \left\{ x \in 2^\w: e(x) \prec o(x) \right\}$.
We show $\Gamma$ is not $\poset{MV}$-measurable.
Given any $p \in \poset{MV}$, let $\{n_k: k \geq 1 \}$ enumerate all natural numbers in $S(p) \cup U(p)$ (note that in the enumeration of the $n_k$'s we drop the index $p$ for making the notation less cumbersome, since the tree $p$ we refer to is fixed). 
To prove our claim, we aim to find $x$, $z \in [p]$ such that $x \in \Gamma \ifif z \notin \Gamma$. 
We pick $x \in [p]$ such that for all $n_k \in S(p) \cup U(p)$, $x(n_k)=1$, i.e. for every $k \geq 1$, $n^x_k=n_k$. 
Let $\left\{\left(n_{m_j}, n_{m_{j}+1}, n_{m_{j}+2}\right): j \geq 1 \right\}$ be an enumeration of all Mathias triples in $p$.  
We need to consider three cases.
\begin{itemize}
\item Case $e(x) \prec o(x)$: 
We remove $n_{m_1+1}$, $n_{m_j}$, $n_{m_j+1}$, for all $j>1$ from $U(x)$ to obtain $z \in 2^\w$ as follows: 
\[
z(n) := \Big \{ 
\begin{array}{ll}
x(n) & \text{if $n\notin \left\{n_{m_1+1}, n_{m_j}, n_{m_j+1}:  j>1\right\}$}\\
0 &  \text{otherwise}.
\end{array}
\]
Note that $z \in [p]$, since $n_{m_1+1}, n_{m_j}, n_{m_{j}+1}$ are all in $S(p)$. Let
\[
\begin{split}
O(m_1):=& [n_1,n_2) \cup [n_3,n_4) \cdots [n_{m_1-1}, n_{m_1}),\\
E(m_1):=& [n_2,n_3) \cup [n_4,n_5) \cdots [n_{m_1}, n_{m_1+1}).
\end{split}
\]
Let $\{k_1, k_2, \cdots, k_M\}$ enumerate the elements in $O(m_1)$, and let $\{k^1, \cdots, k^M\}$ enumerate the initial $M$ elements of the infinite set $\underset{j>1}{\bigcup}[n_{m_j}, n_{m_j+1})$.
We permute $e(z)(k_1)$ with  $e(z)(k^1)$, $e(z)(k_2)$ with  $e(z)(k^2)$, continuing likewise till $e(z)(k_M)$ with  $e(z)(k^M)$ to obtain $e^{\pi}(z)$.
Further, $o^{\pi}(z)$ is obtained by carrying out identical permutation on $o(z)$.
Observe that $e^{\pi}(z)$ and $o^{\pi}(z)$ are finite permutations of $e(z)$ and $o(z)$ respectively.
Then,
\begin{itemize}
\item[-]{for all $n\in O(m_1)$,  $e^{\pi}(z)(n) = 1 = o(x)(n)$ and $o^{\pi}(z)(n) = 0 = e(x)(n)$,}
\item[-]{for all $n\in E(m_1)$,  $e^{\pi}(z)(n) = 1 > 0 = o(x)(n)$ and $o^{\pi}(z)(n) = 0 < 1 = e(x)(n)$,}
\item[-]{for all $n\in \underset{j>1}{\bigcup} [n_{m_j}, n_{m_j+1})\setminus \{k^1, \cdots, k^M\}$,  $e^{\pi}(z)(n) = 1 > 0 = o(x)(n)$ and  $o^{\pi}(z)(n) = 0 <1 = e(x)(n)$,}
\item[-]{for $n\in \{k^1, \cdots, k^M\}$,  $e^{\pi}(z)(n) = 0= o(x)(n)$ and $o^{\pi}(z)(n) = 1= e(x)(n)$, and }
\item[-]{for all remaining  $n\in \w$,  $e^{\pi}(z)(n) = o(x)(n)$ and $o^{\pi}(z)(n) = e(x)(n)$.}
\end{itemize}
Observe that A implies
$e^{\pi}(z)\sim e(z)\; \text{and}\; o^{\pi}(z)\sim o(z)$
and IP implies 
$o(x) \prec e^{\pi}(z) \; \text{and}\; o^{\pi}(z) \prec e(x)$.
Combining them with transitivity, we thus get
$o(z) \sim o^{\pi}(z)\prec e(x) \prec o(x) \prec e^{\pi}(z) \sim e(z)$ and so  $o(z) \prec e(z)$,
which implies $z\notin \Gamma$.

\vspace{2mm}
\item Case $o(x) \prec e(x)$: the argument is similar to the above case and we just need to arrange the details accordingly. We remove $n_{m_1}$, $n_{m_j+1}$, $n_{m_j+2}$, for all $j>1$ from $U(x)$ to obtain $z \in 2^\w$ as follows:
\[
z(n) := \Big \{ 
\begin{array}{ll}
x(n) & \text{if $n\notin \left\{n_{m_1}, n_{m_j+1}, n_{m_j+2}:  j>1\right\}$}\\
0 &  \text{otherwise}.
\end{array}
\]
Let
\[
\begin{split}
O(m_1):=& [n_1,n_2) \cup [n_3,n_4) \cdots [n_{m_1-1}, n_{m_1}),  \\
E(m_1):=& [n_2,n_3) \cup [n_4,n_5) \cdots [n_{m_1-2}, n_{m_1-1}).
\end{split}
\]
(In case $m_1=2$ put $E(m_1)=\emptyset$.)
Let $\{k_1, k_2, \cdots, k_M\}$ enumerate the elements in $E(m_1)$, and let $\{k^1, \cdots, k^M\}$ enumerate the initial $M$ elements of the infinite set $\underset{j>1}{\bigcup}[n_{m_j+1}, n_{m_j+2})$.
We permute $e(z)(k_1)$ with  $e(z)(k^1)$, $e(z)(k_2)$ with  $e(z)(k^2)$, continuing likewise till $e(z)(k_M)$ with  $e(z)(k^M)$ to obtain $e^{\pi}(z)$.
Further, $o^{\pi}(z)$ is obtained by carrying out identical permutation on $o(z)$.
Observe that $e^{\pi}(z)$ and $o^{\pi}(z)$ are finite permutations of $e(z)$ and $o(z)$, respectively.
Then, by arguing as in the previous case, one can
observe that A implies
$e^{\pi}(z)\sim e(z)\; \text{and}\; o^{\pi}(z)\sim o(z)$
and IP gives
$e^{\pi}(z) \prec o(x) \; \text{and}\; e(x)  \prec o^{\pi}(z)$. 
Combining them we obtain
$e(z) \sim e^{\pi}(z) \prec o(x)  \prec e(x) \prec o^{\pi}(z)\sim o(z)$ and so $e(z) \prec o(z)$,
which implies $z\in \Gamma$.

\vspace{2mm}

\item Case $e(x) \sim o(x)$: We remove $n_{m_j}$, $n_{m_j+1}$, for all $j>1$ from $U(x)$ to obtain $z \in 2^\w$ as follows:
\[
z(n) = \Big \{ 
\begin{array}{ll}
x(n) & \text{if $n\notin \left\{n_{m_j}, n_{m_j+1}: j>1\right\}$}\\
0 &  \text{otherwise}.
\end{array}
\]
By construction we obtain $o(z)(n) \geq o(x)(n)$ and $e(z)(n) \leq e(x)(n)$ for all $n\in \w$.
Further, for all $n \in \underset{j\in \w}{\bigcup} \left[n_{m_j}, n_{m_{j}+1}\right)$, $o(z)(n) = 1 > 0 = o(x)(n)$ and $e(z)(n) = 0 < 1 = e(x)(n)$.
Hence, by IP, we get 
$o(x) \prec o(z) \; \text{and}\; e(z) \prec e(x)$,
and so by transitivity it follows
$e(z) \prec  o(z)$, which implies $z\in \Gamma$.
\end{itemize}

\end{proof}

\section{Welfare-regularity Diagram} \label{section2}

The results proved in the previous section, together with the results mentioned in the introduction already discovered by Lauwers and Zame, yield to the following \emph{Welfare-Regularity Diagram} (\emph{WR-diagram}), which essentially represents the fragments of AC corresponding to the non-constructive sets involved in our investigation. Since the utility domain $Y=\{ 0,1 \}$ is fixed, throughout this section, we simply write $\SP$ ($\IP$ resp.) instead of $\SP_{\{0,1\}}$ ($\IP_{\{ 0,1 \}}$ resp.).

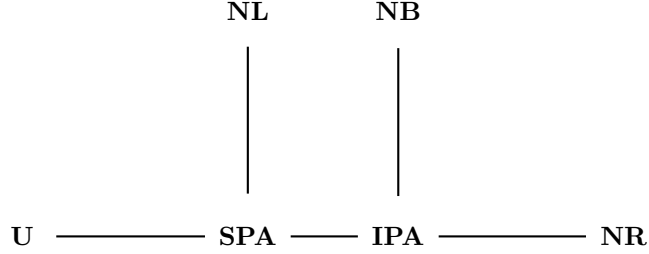
\begin{center}
\begin{tikzpicture}[-, auto, node distance=3.0cm, thick]
\tikzstyle{every state}=[fill=white, draw=none, text=black]
\node[state](U){\U};
\node[state](SP)[right of=U]{\SP};
\node[state](IP)[right of=SP, node distance=2.0cm]{\IP};
\node[state] (NL) [above of=SP]{\NL};  
\node[state] (NB) [above of=IP]{\NB}; 
\node[state](NR)[right of=IP]{\NR};

\path 	(U) edge (SP)
         	(SP) edge  (IP)
		(SP) edge (NL)    
         	(IP) edge (NB)
		(IP) edge (NR)
		;
\end{tikzpicture}
\end{center}

This WR-diagram mimics other popular diagrams in set theory of the reals (like Cicho\'n's diagram) and it should be understood similarly; moving left-to-right or bottom-up means moving from a stronger to a weaker statement (in terms of ZF-implications). As for Cicho\'n's diagram, we want to consider combinations of $\square$'s and $\blacksquare$'s using the following convention: $\square$ means that the corresponding statement is true, $\blacksquare$ means that the corresponding statement is false. 
It is then interesting to understand whether the various ZF-implications do not reverse and, more generally, if provided a combination of $\blacksquare$'s and $\square$'s, one can find the suitable model satisfying such a given combination.

\subsection{A model for $\IP=\blacksquare$, $\NB=\square$, $\NL=\square$, $\NR=\square$}

The proof uses some idea from the previous section on Mathias-Silver trees together with the proof-method used in \cite[Proposition 3.7]{BLH2005}. 

\begin{lemma} \label{L2}
Let $c$ be a Cohen generic real. Then 
\[
V[c] \models \exists q \in \poset{MV} \forall z \in [q] (z \text{ is a Cohen real}).
\]
\end{lemma}
\begin{proof}
It follows the same idea as in the proof of Lemma \ref{l1}. Consider the poset $\poset{F}$ consisting of all  $F \subseteq 2^{<\omega}$ finite uniform trees, i.e., such that:
\begin{itemize}
\item all terminal nodes of $F$ have the same length;
\item $\forall s,t \in F \forall i \in \{ 0,1 \}(|s|=|t| \Rightarrow (s^\conc i \in F \ifif t^\conc i \in F))$.
\end{itemize} 
$\poset{F}$ is ordered by end-extension: $F' \leq F$ iff $F' \supseteq F$ and for all $t \in F' \setminus F$ there is $s \in \term(F)$ such that $s \subseteq t$.
Since $\poset{F}$ is countable (and non-trivial), it is equivalent to $\cohen$. 

Let $G$ be $\poset{F}$-generic over $V$ and put $q_G:= \bigcup G$.
We claim that $q_G \in \poset{MV}$ and every of its branch is Cohen over $V$.
To show that it is sufficient to prove that given any $F \in \poset{F}$ and $D \subseteq \cohen$ open dense from the ground model $V$, 
there exists $F' \leq F$ such that 
\(
F' \force \forall t \in \term(F') (t \in D).
\)
It is easy to see that one can use the same argument as in the proof of Lemma \ref{l1} (actually, in this case one step being sufficient and no need of $\omega$-many steps), by using the same eight sequences $s_1, s_2, \dots, s_8$ in order to make sure that $q_G \in \poset{MV}$ and then uniformly extend the nodes in order to get that for all $t \in \term(F')$, $t \in D$.   
\end{proof}

\begin{proposition}
Let $\cohen_{\omega_1}$ be an $\omega_1$-product of $\cohen$ with finite support and $G$ be $\cohen_{\omega_1}$-generic over $L$. Then
\[
L(\mathbb{R})^{L[G]} \models \neg \mathbf{IPA} \land \mathbf{NB} \land \mathbf{NL} \land \mathbf{NR}.
\]
\end{proposition}
\begin{proof}
It follows the same argument as in the proof of  \cite[Proposition 3.7]{BLH2005}; we give here the proof for completeness and arrange some details according to our setting.  We aim to show that in $L[G]$ for every $\On^\omega$-definable set $X \subseteq 2^\omega$ there exists $q \in \poset{MV}$ such that $[q] \subseteq X$ or $[q] \cap X = \emptyset$. The key idea is that since Cohen forcing adds a generic Mathias-Silver tree of Cohen branches and Cohen forcing is strongly homogeneous, we thus have a factoring lemma \emph{\'a la Solovay}. More precisely, we can argue as follows. Given $X:=\{x \in 2^\omega: \varphi(x,v)  \}$ with $\varphi$ formula and $v \in \On^\omega$, we can use standard argument to absorb $v$ into the ground model, i.e., pick $\alpha<\omega_1$ such that $v \in L[G\restric \alpha]$. Let $\Phi(x,v)$ be the formula asserting that $\force_{\cohen_{\omega_1}} \varphi(x,v)$.

By strong homogeneity of $\cohen_{\omega_1}$, one has the following factoring Lemma: for every $x \in 2^\omega \cap L[G]$, there exists a $\cohen_{\omega_1}$-generic filter $H$ over $L[G \restric \alpha][x]$ such that $L[G]=L[G\restric \alpha][x][H]$.

Then Lemma \ref{L2} gives $q \in L[G \restric \alpha+1]$ such that all $x \in [p]^{L[G]}$ are Cohen over $L[G \restric \alpha]$; moreover note that an easy refinement of the proof-argument indeed shows that we can arbitrarily pick $\stem(q)$ being any $t \in 2^{<\omega}$. Finally, the latter combined with the factoring lemma, gives: for all  $x \in 2^\omega \cap L[G]$, if $x$ is Cohen over $L[G \restric \alpha]$, then 
\[
L[G \restric \alpha][x] \models \Phi(x,v)  \quad \Leftrightarrow \quad L[G] \models \varphi(x,v).
\]
Since $q$ only consists of Cohen branches, and by homogeneity of $\cohen$, we thus obtain 
\[
L[G] \models \forall x \in 2^\omega (x \in [q] \Rightarrow \varphi(x,v)) \quad \text{or} \quad L[G] \models \forall x \in 2^\omega (x \in [q] \Rightarrow \neg \varphi(x,v)).
\]
Finally, from various characterizations proved in \cite{BL99} and some known preservation theorems, it follows that in $L[G]$:
\begin{itemize}
\item there exists a $\SSigma^1_2$ non-Baire set, as $\cohen_{\omega_1}$ does not add a comeager set of Cohen reals (\cite[Theorem 5.8]{BL99} and \cite[Lemma 6.5.3, p. 313]{BJ1995});
\item there exists a $\DDelta^1_2$ non-Ramsey set, as $\cohen_{\omega_1}$ does not add dominating reals  (\cite[Theorem 4.1]{BL99} and \cite[Lemma 6.5.3, p. 313]{BJ1995}: note that a non-Laver measurable set is a special case of a non-Ramsey set);
\item there exists a $\DDelta^1_2$ non-Lebesgue set, as $\cohen_{\omega_1}$ does not add random reals (\cite[Theorem 6.5.28, p. 322]{BJ1995} and \cite[Theorem 9.2.1, p. 452]{BJ1995}). 
\end{itemize}  
Hence, passing into the inner model $L(\mathbb{R})$ of $L[G]$ we have $\neg \IP \land \NB \land \NL \land \NR$. 
\end{proof}

In particular, it follows $L[G]$ is a model for the diagram 

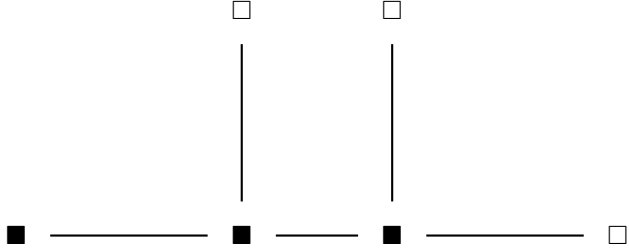
\begin{center}
\begin{tikzpicture}[-, auto, node distance=3.0cm, thick]
\tikzstyle{every state}=[fill=white, draw=none, text=black]
\node[state](U){$\blacksquare$};
\node[state](SP)[right of=U]{$\blacksquare$};
\node[state](IP)[right of=SP, node distance=2.0cm]{$\blacksquare$};
\node[state] (NL) [above of=SP]{$\square$};  
\node[state] (NB) [above of=IP]{$\square$}; 
\node[state](NR)[right of=IP]{$\square$};

\path 	(U) edge (SP)
         	(SP) edge  (IP)
		(SP) edge (NL)    
         	(IP) edge (NB)
		(IP) edge (NR)
		;
\end{tikzpicture}
\end{center}

\subsection{A model for $\NL=\blacksquare$, $\NR=\square$} 
We recall the following well-known forcing notions, which we use through this section.
\begin{itemize}
\item Random forcing $\random:=\{ C \subseteq 2^{\omega}: C \text{ closed } \land \mu(C) > 0 \}$, where $\mu$ is the standard Lebesgue measure on $2^\omega$. The order is given by: $C' \leq C \ifif C' \subseteq C$.
\item Mathias forcing $\mathias$ consisting of pairs $(s,x)$ such that $x \in [\omega]^\omega$, $s \in [\omega]^{<\omega}$ and $\max s <  \min x$, ordered by $(t,y) \leq (s,x)$ iff $t \supseteq s$, $t \restric |s| = s$ and $y \subseteq x$. Moreover we denote 
\[
[s,x]:= \{y \in [\omega]^\omega: y \supset s \land y \restric |s|= s \land y \subseteq x  \}.
\]
\item Given $\kappa > \omega$ cardinal, let 
\[
\mathsf{Fn}(\omega,\kappa):= \{ f: f \text{ is a function} \land |\dom(f)| < \omega \land \dom(f) \subseteq \omega \land \text{ran} (f) \subseteq \kappa  \}, 
\]
ordered by: $f' \leq f \ifif f' \supseteq f$. Note $\mathsf{Fn}(\omega,\kappa)$ is the standard poset adding a surjection $f_G: \omega \rightarrow \kappa$, i.e., the forcing collapsing $\kappa$ to $\omega$. 
\end{itemize}

\begin{theorem} \label{thm1}
There is a ZF-model $N$ such that
\[
N \models  \nr \land \neg \nl.
\]
\end{theorem}

The model $N$ is going to be the inner model of a certain forcing extension that we are going to define in the proof of Theorem \ref{thm1} below. 
The key idea to obtain such a forcing-extension is to use Shelah's amalgamation over random forcing with respect to a certain name $Y$ for sets of elements in $2^\omega$ in order to get a complete Boolean algebra $B$ such that, if $G$ is $B$-generic over $V$, in $V[G]$ the following hold:
\begin{enumerate}
\item every subset of $2^\omega$ in $L(\real,\{ Y \})$ is Lebesgue measurable
\item $Y$ is non-Ramsey.
\end{enumerate}
Hence, we obtain that in $L(\real,\{Y\})^{V[G]}$ every subset of $2^\omega$ is Lebesgue measurable (and so by Zame's result there cannot be any total SWR satisfying A and SP), but there is a non-Ramsey set. 

Shelah's amalgamation (\cite{Sh85}) is the main tool we need for our forcing construction. Since it is a rather demanding machinery, we refer the reader to the Appendix for a more detailed approach and an exposition of the main properties.
The reader already familiar with Shelah's amalgamation can proceed with no need of such Appendix. 

Before going to the detailed and technical proof, we just give a short overview of the proof-structure.
Starting from a Boolean algebra $B$, two complete subalgebras $\mathbb{B}_0, \mathbb{B}_1 \lessdot B$ isomorphic to the random algebra with $\phi$ isomorphism between them, the amalgamation process provides us with the pair $(B^*,\phi^*)$ such that $B \lessdot B^*$ and $\phi^* \supseteq \phi$ such that $\phi^*$ is an automorphism of $B^*$. We denote this amalgamation process by $\amal^\omega(B,\phi)$, so that $B^*=\amal^\omega(B,\phi)$.

Since the process itself generates more and more copies of random algebras, we have to iterate this process as long as we treat all of the copies of such random algebras. For doing that a recursive book-keeping argument of length $\kappa$ inaccessible will be sufficient (and necessary to ensure the final construction satisfy $\kappa$-cc).

The idea to obtain 1 and 2 above is based on the following main parts:
\begin{itemize}
\item[(a)] The Boolean algebra $B$ is built via a recursive construction, alternating the amalgamation, iteration with $\mathsf{Fn}$, iteration with Mathias forcing and picking direct limits at limit steps. 
\item[(b)] The set $Y$ is also recursively built by carefully adding Mathias reals cofinally often in order to get a non-Ramsey set.
\item[(c)] In order to obtain that all sets of reals in $L(\real, \{ Y \})$ be Lebesgue measurable, we have to amalgamate over random forcing, and we also need to recursively close $Y$ under the isomorphisms between copies of the random algebras generated by the amalgamation process, in order to get $\force \phi[Y]=Y$, for every such isomorphism $\phi$.
\end{itemize}

In particular to get (c) the algebra $B$ we are going to construct is going to satisfy \emph{$(\random,Y)$-homogeneity}, i.e., for every pair of random algebras $\poset{B}_0,\poset{B}_1 \lessdot B$ with $\phi: \poset{B}_0 \rightarrow \poset{B}_1$ isomorphism, there exists $\phi^* \supseteq \phi$ automorphism of $B$ such that $\force_B \phi^*[Y]=Y$. (Roughly speaking: any isomorphism between copies of random algebra can be extended to an automorphism which fixes $Y$). See \cite[Theorem 6.2.b]{JR93} for a proof that $(\random,Y)$-homogeneity is the crucial ingredient to force that all sets in $L(\real,\{ Y \})$ are Lebesgue measurable. 

We now see the construction of the complete Boolean algebra $B$ and the proof of Theorem \ref{thm1} in detail.

\begin{proof}[Proof of Theorem \ref{thm1}]
Start from a ground model $V$ we are going to recursively define $\{B_\alpha: \alpha < \kappa  \}$ sequence of complete Boolean algebras such that $B_\alpha \lessdot B_\beta$, for $\alpha < \beta$, and $\{ Y_\alpha: \alpha <\kappa \}$ $\subseteq$-increasing sequence of sets of names for reals, and then put $B:= \lim_{\alpha < \kappa} B_\alpha$ and $Y:= \bigcup_{\alpha<\kappa} Y_\alpha$. The construction follows the line of the one presented in \cite{JR93}, even if it sensitively differs in the construction of the set $Y$and in proving that it is non-Ramsey, instead of a set without the Baire property. We also use the forcing $\mathsf{Fn}$ instead of the amoeba for measure, as it also serves the scope of collapsing the additivity of the null ideal and to ensure the inaccessible $\kappa$ be gently collapsed to $\omega_1$ in the forcing-extension via $B$.  We start with $B_0 = \{ 0 \}$ and $Y_0=\emptyset$. 

\begin{enumerate}
 \item In order to obtain the $(\random,\dot{Y})$-homogeneity we use a standard book-keeping argument to hand us down all possible situations of the following type:
if $\algebra{B}_\alpha \lessdot \algebra{B}' \lessdot \algebra{B}$ and $\algebra{B}_\alpha \lessdot \algebra{B}'' \lessdot \algebra{B}$ are such that
$\algebra{B}_\alpha$ forces $(\algebra{B}'/\algebra{B}_\alpha) \approx (\algebra{B}''/\algebra{B}_\alpha) \approx \random$
and $\phi_0: \algebraQ' \rightarrow \algebraQ''$ an isomorphism s.t. $\phi_0 {\upharpoonright} \algebra{B}_\alpha= \text{Id}_{\algebra{B}_\alpha}$, then there exists
a sequence of functions in order to extend the isomorphism $\phi_0$ to an automorphism $\phi: \algebra{B}
\rightarrow \algebra{B}$, i.e., $\exists \langle \alpha_\eta : \eta
< \kappa \rangle$ increasing, cofinal in $\kappa$, with $\alpha_0=\alpha$, and $ \exists \langle
\phi_\eta : \eta < \kappa \rangle$ such that 
\begin{itemize}
\item for $\eta >0$ successor ordinal, $\algebra{B}_{\alpha_{\eta}+1} :=\amal^\omega(\algebra{B}_{\alpha_{\eta}},\phi_{\eta-1})$,
and $\phi_\eta$ be the automorphism on  $\algebra{B}_{\alpha_\eta+1}$ generated by the amalgamation;
\item  for $\eta$ limit ordinal, let $\algebra{B}_{\alpha_\eta}:= \lim_{\xi < \eta} \algebra{B}_{\alpha_\xi}$ and $\phi_\eta= \lim_{\xi < \eta} \phi_\xi$, in the obvious sense;
\item for every $\eta< \kappa$, we have $\algebra{B}_{\alpha_{\eta}+1} \lessdot \algebra{B}_{\alpha_{\eta+1}}$.
\end{itemize}
In order to fix the set of names by each automorphism $\phi_\eta$, one then sets
\begin{itemize} 
\item successor case $\eta>0$:
\[
\begin{split}
B_{\alpha_{\eta}+1} \force Y_{\alpha_{\eta} +1} &:= Y_{\alpha_{\eta}} \cup \{
\phi^j_{\eta}(\dot{y}), \phi^{-j}_{\eta}(\dot{y}): \dot{y} \in
Y_{\alpha_{\eta}}, j \in \omega \}, \\
\end{split}
\]
\item limit case: $B_{\alpha_\eta} \force Y_{\alpha_\eta} := \bigcup_{\xi < \eta}  Y_{\alpha_\xi}$.
\end{itemize}
\item In order to get $Y$ being non-Ramsey, for cofinally many $\alpha$'s, put $\algebra{B}_{\alpha + 1}:=
\algebra{B}_\alpha * \dot{\mathias}$ and
\[
B_{\alpha+1} \force Y_{\alpha +1}:= Y_\alpha \cup \{
\dot{y}_{(s,x)}: (s,x) \in \mathias \},
\]
where $\dot{y}_{(s,x)}$ is a name for a Mathias real over
$V^{\algebra{B}_\alpha}$ such that $(s,x) \force s \subset \dot{y}_{(s,x)} \subseteq x$.
\item For cofinally many $\alpha$'s pick a cardinal $\lambda_\alpha < \kappa$ such that $B_\alpha \force \lambda_\alpha > \omega$, put $\algebra{B}_{\alpha + 1}:=
\algebra{B}_\alpha * \mathsf{Fn}(\omega, \lambda_\alpha)$, and let
$B_{\alpha+1} \force Y_{\alpha +1}:= Y_\alpha$, where $\mathsf{Fn}(\omega, \lambda_\alpha)$ is the forcing adding a surjection $F_\alpha:  \omega \rightarrow \lambda_\alpha$.
\item For any limit ordinal, put
$\algebra{B}_\lambda := \lim_{\alpha < \lambda} \algebra{B}_\alpha$ and $B_\lambda \force Y_\lambda := \bigcup_{\alpha < \lambda} Y_\alpha$.
\end{enumerate}

Let $G$ be $B$-generic over $V$.
As mentioned above, the proof of ``every set of reals in $L(\real,Y)$ is Lebesgue measurable'' is a standard Solovay-style argument, and can be found in \cite{JR93}. The only difference we adopt is the use of $\mathsf{Fn}(\omega, \lambda_\alpha)$. i.e. the poset ``collapsing" $\lambda_\alpha$ to $\omega$  instead of the amoeba for measure. The property needed for our purpose, which is to turn the union of all Borel null sets coded in the ``ground model" $V[G \restric \alpha +1]$ into a null set, is fulfilled by $\mathsf{Fn}(\omega, \lambda_\alpha)$  as well, i.e. 
\[
\mathsf{Fn}(\omega, \lambda_\alpha) \force \bigcup \{N_c: c \text{ is a Borel code for a null set in $V[G \restric \alpha +1]$}\} \text{ is null},
\]
where $N_c \subseteq 2^\omega$ is the Borel null set coded by $c$.

What is left to show is that 
\begin{equation} \label{eq-non-ramsey}
B \force \text{``$Y$ is non-Ramsey''}.
\end{equation}

For proving that, pick arbitrarily $(s,x) \in \mathias$; we have to show $$B \force Y \cap [s,x] \neq \emptyset \text{ and } [s,x] \not \subseteq Y.$$

For the former, Let $\dot{(s,x)}$ be a
$\algebra{B}$-name for a Mathias condition. By $\kappa$-cc and part (2) of the recursive construction, there is $\alpha < \kappa$
such that $\dot{(s,x)}$ is a $\algebra{B}_\alpha$-name,
$\algebra{B}_{\alpha +1}=\algebra{B}_\alpha
* \dot{\mathias}$ and
$B_{\alpha+1} \force Y_{\alpha+1}=Y_\alpha \cup \{
\dot{y}_{(s,x)}: (s,x) \in \mathias^{B_\alpha} \}$. Consider $\dot{y}_{(s,x)}$ name for a Mathias
real over $V^{B_\alpha}$ such that $B_{\alpha+1} \force \dot{y}_{(s,x)} \in [s,x]$. Thus,
\[
B \force \dot{y}_{(s,x)} \in Y \cap [s,x].
\]

On the other hand, by part (3) of the construction, there is also $\alpha < \kappa$, such that
$\dot{(s,x)}$ is a $\algebra{B}_\alpha$-name,
$\algebra{B}_{\alpha+1}= \algebra{B}_{\alpha}* \mathsf{Fn}(\omega,\lambda_\alpha)$
and $B_{\alpha+1} \force Y_{\alpha+1}=Y_\alpha$. Let
$\dot{y}$ be a $B_{\alpha+1}$-name for a Mathias real
over $V^{B_\alpha}$ such that
$B \force \dot y \in [s,x]$. Obviously, $B \force
\dot{y} \notin Y_{\alpha}$ (since ``the real $y$ is added at
stage $\alpha+1$''), and hence
\[
B \force \dot{y} \in [s,x] \setminus
Y_{\alpha+1},
\]
since
$B \force Y_{\alpha+1}=Y_{\alpha}$.
So it is left to show that also for every $\beta>\alpha+1$, $B \force y \notin Y_{\beta} \setminus Y_{\alpha}$, which means, intuitively speaking, $y$ cannot fall into $Y$ at any later stage $\beta>\alpha+1$.
 For proving that we show the following Claim \ref{claim-mathias}. Fix the notation: given $x \in 2^\omega$, we denote by $f_x$ the increasing enumeration of the set $\{ n \in \omega: x(n)=1 \}$. It is well-known (and straightforward to check) that if $x$ is a Mathias real over $V$, then $f_x$ is dominating over $V \cap \omega^\omega$.
\begin{claim} \label{claim-mathias}
 For $\beta < \kappa$, $\beta > \alpha+1$ and $\dot{y} \in Y_\beta \setminus Y_{\alpha+1}$, one has
\[
B \force \text{``}f_{\dot{y}} \text{ is dominating over } V^{B_{\alpha+1}} \cap \omega^\omega \text{''}.
\]
\end{claim}

For $\beta$ limit the proof is trivial. For $\beta+1$, we have two cases.

Case as in part (2) of the recursive construction, i.e. $Y_{\beta+1}= Y_\beta \cup \{\dot{y}_{(s,x)}: (s,x) \in \mathias \}$. In this case $\dot{y}$ has to be a Mathias real over $V^{B_{\alpha+1}}$ and therefore $f_{\dot y}$ is dominating over $V^{B_{\alpha+1}} \cap \omega^\omega$. 

Case as in part (1) of the construction, i.e. 
\[
B_{\beta+1} \force Y_{\beta +1} := Y_{\beta} \cup \{
\phi^j(\dot{y}), \phi^{-j}(\dot{y}): \dot{y} \in
Y_{\beta}, j \in \omega \}, 
\]
where $\phi$'s are the associated automorphisms generated by the amalgamation. 

The aim is to show that the property of ``being dominating" is preserved through the construction unfolded in part (1), both by the amalgamation process and by iteration of random forcing. More precisely, we need the following lemma.

\begin{lemma} \label{lemma:preserve-dominating}
Let $\eta > 0$ be a successor ordinal.  Let $\algebra{B}', \algebra{B}'' \lessdot \algebra{B}_{\alpha_\eta}$ and $\dot{x} \in V^{\algebra{B}_{\alpha_\eta}} \cap \cantor$
such that
\[
 \algebra{B}_{\alpha_\eta} \force \text{``$f_{\dot{x}}$ is dominating over both $V^{\algebra{B}'} \cap \omega^\omega$ and $V^{\algebra{B}''} \cap \omega^\omega$''},
\]
and $\psi: \algebra{B}' \rightarrow \algebra{B}''$ isomorphism.

Then, for every $j \in \omega$,
\[
 \algebra{B}_{\alpha_\eta+1} \force \text{``$f_{\phi^j_\eta (\dot{x})}$ and $f_{\phi^{-j}_\eta (\dot{x})}$ are dominating over $V^{\algebra{B}_{\alpha_\eta}} \cap \omega^\omega$''}.
\]
where $\algebra{B}_{\alpha_\eta+1}=\amal^\omega(\algebra{B}_{\alpha_\eta}, \psi)$, and $\phi_\eta$ is the automorphism extending $\psi$, generated
by the amalgamation.
\end{lemma}

\begin{sublemma}
[Preservation by one-step amalgamation]  \label{sublemma-1}
Let $\algebra{B}, \algebra{B}_1, \algebra{B}_2, \phi_0$, $e_1$, $e_2$ as in the Appendix and $\dot{x}$ a $\algebra{B}$-name for an element of $\cantor$ such that 
$\algebra{B}$ forces $f_{\dot{x}}$ is dominating over $V^{\algebra{B}_1} \cap \omega^\omega$ and  $V^{\algebra{B}_2} \cap \omega^\omega$.
Then
\begin{equation} \label{eq-5}
\amal(\algebra{B},\phi_0) \force \text{``$f_{e_1(\dot{x})}$ is dominating over $V^{e_2[\algebra{B}]} \cap \omega^\omega$''}.
\end{equation}
(And analogously $\amal(\algebra{B},\phi_0) \force \text{``$f_{e_2(\dot{x})}$ is dominating over $V^{e_1[\algebra{B}]} \cap \omega^\omega$''}$.)
\end{sublemma}

\begin{proof}[Proof of Sublemma 1]
By Lemma \ref{lemma-amal-product} in Appendix, putting $V=N[H]$, $A_1= (B / B_1)^H$, $A_2= (B / B_2)^H$, it is sufficient to prove that given $A_1,A_2$ complete Boolean algebras and $\dot{f}$ a $A_1$-name for an element of $\omega^\omega$, if
$$A_1 \force \text{``$\dot{f}$ is dominating over $V \cap \omega^\omega$''},$$
then
\[
A_1 \times A_2 \force \text{``$\dot{f}$ is dominating over $V[G] \cap \omega^\omega$'' },
\]
where $G$ is $A_2$-generic over $V$.
To reach a contradiction, assume there is $z \in \baire \cap V[G]$, $(a_1,a_2) \in A_1 \times A_2$ such that $(a_1,a_2) \force \exists^\infty n \in \omega ( z(n) > f(n))$. Let $\{ n_j: j \in \omega \}$ enumerate all such $n$'s, and for every $j \in \omega$ pick $b_j \in A_2$, $b_j \leq a_2$ and $k_j \in \omega$ such that $(a_1,b_j) \force z(n_j)=k_j$; note that this can be done since $z \in V[G]$ and $G$ is $A_2$-generic over $V$; hence $z$ can be seen as an $A_2$-name and so it is suffcient to strengthen conditions in $A_2$ in order to decide its values. Since $A_1$ forces $f$ be dominating over $V \cap \omega^\omega$, one can pick $a \leq a_1$ such that $(a,a_2) \force \exists m \forall j \geq m(k_j \leq f(n_j))$. Pick $j' > m$; then 
\begin{itemize} 
\item[-] on the one side, since $(a,b_{j'}) \leq (a_1,a_2)$, it follows $(a,b_{j'}) \force f(n_{j'}) < k_{j'} = z(n_{j'})$ 
\item[-] on the other side, since $(a,b_{j'}) \leq (a,a_2)$, it follows $(a,b_{j'}) \force f(n_{j'}) \geq k_{j'}=z(n_{j'})$,  
\end{itemize}
which is a contradiction.
\end{proof}

\begin{sublemma}[Preservation by $\omega$-step amalgamation] \label{sublemma1bis}
Let $B$ be a complete Boolean algebra, $\algebra{B}', \algebra{B}'' \lessdot \algebra{B}$ and $\dot{x} \in V^{\algebra{B}} \cap \cantor$
such that
\[
\algebra{B} \force \text{``$f_{\dot{x}}$ is dominating over both $V^{\algebra{B}'} \cap \omega^\omega$ and $V^{\algebra{B}''} \cap \omega^\omega$''},
\]
with $\psi: \algebra{B}' \rightarrow \algebra{B}''$ isomorphism.

Then, for every $j \in \omega$,
\[
\amal^\omega(\algebra{B},\psi) \force \text{``$f_{\phi^{j} (\dot{x})}$ and $f_{\phi^{-j} (\dot{x})}$ are dominating over $V^{\algebra{B}} \cap \omega^\omega$''}.
\]
where $\phi: \amal^\omega(\algebra{B},\psi) \rightarrow \amal^\omega(\algebra{B},\psi)$ is the automorphism extending $\psi$, generated
by the amalgamation.
\end{sublemma}

The proof simply consists of a recursive application of Sublemma \ref{sublemma-1} following the line of the proof of \cite[Lemma 3.4]{JR93} by replacing the notion of ``unbounded" with ``dominating". 

Note that \ref{sublemma1bis} is enough to show Lemma \ref{lemma:preserve-dominating} when $\eta \geq 2$ successor, by considering $B=B_{\alpha_\eta}$, $\amal^\omega (B, \phi)=B_{\alpha_{\eta}+1}$, $B'=B_{\alpha_{\eta-1}}$, $B''=\phi_{\eta-1}[B_{\alpha_{\eta-1}}]$ and $\psi=\phi_{\eta-1}$.

 It is only left to show the case $\eta=1$, which is: $\algebra{B}_{\alpha_0} \lessdot \algebra{B}', \algebra{B}'' \lessdot \algebra{B}_{\alpha_1}$
such that
$\algebra{B}_{\alpha_0}$ forces $(\algebra{B}'/ \algebra{B}_{\alpha_0}) \approx (\algebra{B}'' / \algebra{B}_{\alpha_0}) \approx \random$, 
and $\phi_0: \algebra{B}' \rightarrow \algebra{B}''$ isomorphism such that
$\phi_0 \restric \algebra{B}_{\alpha_0} = \emph{Id}_{\algebra{B}_{\alpha_0}}$. Then for every $\dot{x} \in V^{\algebra{B}_{\alpha_1}} \cap \cantor$ such that
$\algebra{B}_{\alpha_1} \force \text{``$f_{\dot{x}}$ is dominating over $V^{\algebra{B}_{\alpha_0}} \cap \omega^\omega$''} $, one has, for every $j \in \omega$,
\[
 \algebra{B}_{\alpha_1+1} \force \text{`` $f_{\phi^j_1 (\dot{x})}$ and $f_{\phi^{-j}_1 (\dot{x})}$ are dominating over $V^{\algebra{B}_{\alpha_1}} \cap \omega^\omega$''}.
\]

But, since $\algebra{B}_{\alpha_0}$ forces both $(\algebra{B}'/\algebra{B}_{\alpha_0}) \approx (\algebra{B}''/\algebra{B}_{\alpha_0}) \approx \random$, by Sublemma \ref{sublemma-1} and the fact that random forcing is $\omega^\omega$-bounding (and thus it preserves dominating reals), we obtain $\amal^\omega(B_{\alpha_1}, \phi_0) = \algebra{B}_{{\alpha_1}+1}$ and
\[
\algebra{B}_{{\alpha_1}+1} \force \text{``$f_{\dot{x}}$ is dominating over both $V^{\algebra{B}_{\alpha_0}*(\algebra{B}'/ \algebra{B}_{\alpha_0})} \cap \omega^\omega$ and
$V^{\algebra{B}_{\alpha_0}*(\algebra{B}''/ \algebra{B}_{\alpha_0})} \cap \omega^\omega$''}.
\]
\end{proof}

It is easy to see that the construction developed can be combined with Shelah's original one, simply by recursively construct in parallel a set being non-Baire. As a consequence one can obtain a model satisfying the following WR-Diagram
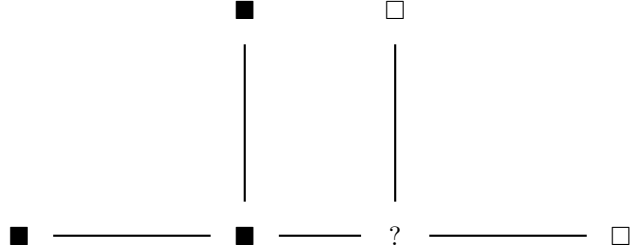
\begin{center}
\begin{tikzpicture}[-, auto, node distance=3.0cm, thick]
\tikzstyle{every state}=[fill=white, draw=none, text=black]
\node[state](U){$\blacksquare$};
\node[state](SP)[right of=U]{$\blacksquare$};
\node[state](IP)[right of=SP, node distance=2.0cm]{?};
\node[state] (NL) [above of=SP]{$\blacksquare$};  
\node[state] (NB) [above of=IP]{$\square$}; 
\node[state](NR)[right of=IP]{$\square$};

\path 	(U) edge (SP)
         	(SP) edge  (IP)
		(SP) edge (NL)    
         	(IP) edge (NB)
		(IP) edge (NR)
		;
\end{tikzpicture}
\end{center}

Note that the status of $\IP$ is not clear in this model.

\begin{remark}
Some other combinations of the WR-Diagram are already known or follow easily from known results. 
For instance, in order to obtain a model for $\NB=\blacksquare \land \NL=\square$, we can consider $N$ be Shelah's model constructed in \cite{Sh85}, where every set of reals has the Baire property. Note that such a model is obtained with no need of inaccessible cardinals. Since in \cite{Sh85} it is also shown that to get a model where every set of reals is Lebesgue measurable we need an inaccessible, we can then deduce that in $N$ there is a set that is not Lebesgue measurable. Note that in such a model the status of $\NR$ is not clear. More generally, the interplay between $\NB$ and $\NR$ is still open, since the lemmata about the preservation of dominating and/or unbounded reals do not extend when we amalgamate over Cohen or Mathias forcing, in place of random forcing.
\end{remark}

\subsection{Weak Pareto for larger utility domain}
\label{s4.2}

In this last sub-section we make a digression away from the WR-Diagram and we dealt with WP, thus providing an answer to \cite[Problem 11.14]{Mathias} also in case we consider the Paretian condition being the weakest possible. As we already notice in Remark \ref{Remark2}, WP is trivial when $Y=\{ 0,1 \}$. Moreover, whenever $Y$ is well-founded, then one can simply consider the function $f: Y^\omega \rightarrow \mathbb{R}$ such that $f(x):= \min\{ x(n): n \in \omega\}$; then define $x \prec y :\Leftrightarrow f(x) < f(y)$ and $x \sim y :\Leftrightarrow f(x) = f(y)$ in order to get a total SWR on $Y^\omega$ satisfying WP and A. 

Here we give a proof that the existence of a total SWR satisfying WP and A gives a non-$\poset{MV}$-measurable set when $Y \subseteq [0,1]$ contains a subset with order type $\mathbb{Z}$; we present a proof for $Y=\mathbb{Z}$ to make the notation less cumbersome, but it is straightforward to notice that precisely the same argument works for any $Y$ with order type $\mathbb{Z}$.

Given $x \in 2^\w$, let $U(x):= \{n \in \w: x(n)=1\}$ and $\{n^x_k: k \in \w \}$ enumerate $U(x)$. As in the case of Proposition \ref{P-ip}, 
define $o(x)$ and $e(x)$; next use the following notation:
\begin{itemize}
\item let $\{l_k: k \geq 1\}$ enumerate all elements in $o(x)$ and $\{u_k: k \geq 1\}$ enumerate all elements in $\omega \setminus o(x)$;
\item let $\{l^{\prime}_k: k \geq 1\}$ enumerate all elements in $e(x)$ and $\{u^{\prime}_k: k \geq 1\}$ enumerate all elements in $\omega \setminus e(x)$;
\end{itemize}
Note that for every $k \geq 1$, one has $l^{\prime}_k=u_{n_1+(k-1)}$ and $l_k=u'_{n_1+(k-1)}$.
Next we define the following pair of sequences $o(\textbf{x}), e(\textbf{x})$ in $\mathbb{Z}^\omega$:
\begin{equation}\label{T3e0}
o(\textbf{x})(n)= \Big \{ 
\begin{array}{ll}
k & \text{if $n=l_k$,  for some $k \geq 1$}\\
-k& \text{if $n=u_k$, for some $k \geq 1$}, 
\end{array}
\end{equation}

\begin{equation}\label{T3e00}
e(\textbf{x})(n) = \Big \{ 
\begin{array}{ll}
 k & \text{if $n=l^{\prime}_k$, for some $k \geq 1$}\\
-k& \text{if $n=u^{\prime}_k$, for some $k \geq 1$}. 
\end{array}
\end{equation}

\begin{proposition}\label{P-wp} 
Let $\precsim$ denote a total SWR satisfying WP and A on $X= \mathbb{Z}^{\w}$.
Then there exists a subset of $2^{\w}$ which is not $\poset{MV}$-measurable.%
\footnote{It is clear from the proof that one could get the same result in an even slightly more general setting, namely with $Y$ any set of utilities with order type $\mathbb{Z}$.}
\end{proposition}

\begin{proof}
The structure of the proof is similar to Proposition \ref{P-ip}, but some technical details are different.
Let $\precsim$ be a total SWR satisfying WP and A, and put $\Gamma:= \{x \in 2^\w: e(\textbf{x}) \prec o(\textbf{x})\}$. 
Given any $p \in \poset{MV}$, let $\{n_k: k \geq 1 \}$ enumerate all natural numbers in $S(p) \cup U(p)$. 
We aim to find $x$, $z \in [p]$ such that $x \in \Gamma \ifif z \notin \Gamma$. 
We proceed as follows: pick $x \in [p]$ such that for all $n_k \in S(p) \cup U(p)$, $x(n_k)=1$. 
Let $\left\{\left(n_{m_j}, n_{m_{j}+1}, n_{m_{j}+2}\right): j \in \w \right\}$ be an enumeration of all Mathias triples in $p$. 
We need to consider three cases.

\begin{itemize}
\item Case $e(\textbf{x}) \prec o(\textbf{x})$: We remove $n_{m_1+1}$, $n_{m_j}$, $n_{m_j+1}$, for all $j>1$ from $U(x)$ to obtain $z \in 2^\w$ as follows: 
\[
z(n) = \Big \{ 
\begin{array}{ll}
x(n) & \text{if $n\notin \left\{n_{m_1+1}, n_{m_j}, n_{m_j+1}: j>1\right\}$}\\
0 &  \text{otherwise}.
\end{array}
\]
Let
\[
\begin{split}
O(m_1):=& [n_1,n_2) \cup [n_3,n_4) \cdots [n_{m_1-1}, n_{m_1}), \\
E(m_1):=& [n_2,n_3) \cup [n_4,n_5) \cdots [n_{m_1}, n_{m_1+1}).
\end{split}
\]

\begin{claim}\label{C1}
There exists $N\in \w$ such that $e(\textbf{x})(n)> o(\textbf{z})(n)$ holds for all $n>N$.
\end{claim}
\begin{proof}
We distinguish two cases.
\begin{enumerate}
\item{$|O(m_1)| < |E(m_1)|$: Among coordinates $n<n_{m_1+1}$, 
\begin{itemize}
\item fewer negative integers have been assigned in $e(\textbf{x})(n)$ as compared to $o(\textbf{z})(n)$.
Then $0> e(\textbf{x})(n_{m_1+1}) > o(\textbf{z})(n_{m_1+1})$ and for all subsequent coordinates $n$ with both $e(\textbf{x})(n)$ and $o(\textbf{z})(n)$ being negative, $0> e(\textbf{x})(n) > o(\textbf{z})(n)$ holds.
\item fewer positive integers have been assigned in $o(\textbf{z})(n)$ as compared to $e(\textbf{x})(n)$.
Then $e(\textbf{x})(n_{m_1+2}) > o(\textbf{z})(n_{m_1+2})>0$ and for all subsequent coordinates $n$ with both $e(\textbf{x})(n)$ and $o(\textbf{z})(n)$ being positive, $e(\textbf{x})(n) > o(\textbf{z})(n)>0$ holds.
\end{itemize}
We take $N = n_{m_1+1}$ in this case.}
\item{$|O(m_1)| \geq |E(m_1)|$: Among the coordinates $[n_{m_j+1}, n_{m_{j+1}})$ for all $j\in \w$,  $e(\textbf{x})(n)$ and  $o(\textbf{z})(n)$ contain equally many elements of same sign.
Further for the coordinates in $[n_{m_j}, n_{m_{j}+1})$, $o(\textbf{z})(n)$ is negative but $e(\textbf{x})(n)$ is not.
Thus for some $J\in \w$, 
\[
|O(m_1)| < |E(m_1)| + \left|\underset{j\in \{2, \cdots, J\}}{\bigcup} \left[n_{m_j}, n_{m_j+1}\right)\right|
\] 
will be true.
In this case, we can apply argument of case (i) above for $n_{m_J+1}$ and therefore obtain $N =n_{m_J+1}$.}
\end{enumerate}
Thus we have shown that for all $n>N$, if $e(\textbf{x})(n)$ and $o(\textbf{z})(n)$ share the same sign then $e(\textbf{x})(n)>o(\textbf{z})(n)$.
The remaining situation is $e(\textbf{x})(n)>0>o(\textbf{z})(n)$.
This completes the proof.
\end{proof}

\begin{claim}\label{C2}
There exists a finite permutation $o^{\pi}(\textbf{z})$ of $o(\textbf{z})$ such that $e(\textbf{x})(n)> o^{\pi}(\textbf{z})(n)$ holds for all $n\in \w$.
\end{claim}

\begin{proof}
In claim \ref{C1}, it has been shown that for all $n>N$ $e(\textbf{x})(n)> o(\textbf{z})(n)$.
Let $K:= \{k^0, k^1, \cdots, k^N\}$ be an increasing enumeration of all elements from  the set $\underset{j>J}{\bigcup} \left[n_{m_j}, n_{m_j+1}\right)$.
We permute  $o(\textbf{z})(0)$ and $o(\textbf{z})(k^0)$; $o(\textbf{z})(1)$ and $o(\textbf{z})(k^1)$ and so on till $o(\textbf{z})(N)$ and $o(\textbf{z})(k^N)$ to obtain $o^{\pi}(\textbf{z})$.
Hence, $o^{\pi}(\textbf{z})$ is obtained via a finite permutation of $o(\textbf{z})$.
It is immediate to check that $\pi$ has the desired properties.
\end{proof}

Applying Claims \ref{C1} and \ref{C2}, we have obtained $o^{\pi}(\textbf{z})$ such that 
$e(\textbf{x})(n)>o^{\pi}(\textbf{z})(n)$ for all $n\in \w$.
A implies 
$o(\textbf{z})\sim o^{\pi}(\textbf{z})$, and by WP we get 
$e(\textbf{x})\succ o^{\pi}(\textbf{z})$.
By applying transitivity, we obtain
$e(\textbf{x})\succ o(\textbf{z})$.

Notice that arguments of claims \ref{C1} and \ref{C2} could also be applied to the pair of sequences $e(\textbf{z})$ and $o(\textbf{x})$.
Thus we are able to obtain $o^{\pi}(\textbf{x})$ such that applying A we get 
$
o(\textbf{x})\sim o^{\pi}(\textbf{x}),
$
by WP we get 
$
o^{\pi}(\textbf{x})\prec e(\textbf{z}),
$
and finally, by transitivity it follows
$
o(\textbf{x})\prec e(\textbf{z}).
$
Combining all together we obtain
$o(\textbf{z}) \prec e(\textbf{x}) \prec o(\textbf{x}) \prec e(\textbf{z})$, and so $o(\textbf{z}) \prec e(\textbf{z})$,
which implies
$ 
z\notin \Gamma.
$

\vspace{2mm}

\item Case $o(\textbf{x}) \prec e(\textbf{x})$: Similar to the previous case, only with some different technical details. We remove $n_{m_1}$, $n_{m_j+1}$, $n_{m_j+2}$, for all $j>1$ from $U(x)$ to obtain $z \in 2^\w$ as follows:
\[
z(n) = \Big \{ 
\begin{array}{ll}
x(n) & \text{if $n\notin \left\{n_{m_1}, n_{m_j+1}, n_{m_j+2}: j>1\right\}$}\\
0 &  \text{otherwise}.
\end{array}
\]
Let
\[
\begin{split}
O(m_1):=& [n_1,n_2) \cup [n_3,n_4) \cdots [n_{m_1-1}, n_{m_1}),  \\
E(m_1):=& [n_2,n_3) \cup [n_4,n_5) \cdots [n_{m_1-2}, n_{m_1-1}).
\end{split}
\]
(In case $m_1=2$ put $E(m_1)=\emptyset$.)

Applying Claims \ref{C1} and \ref{C2}, we are able to obtain $e^{\pi}(\textbf{z})$ and  $o^{\pi}(\textbf{z})$ such that 
\(
o(\textbf{x})(n)>e^{\pi}(\textbf{z})(n), \;\text{and}\; o^{\pi}(\textbf{z})(n)>e(\textbf{x})(n)\;\text{for all}\; n\in \w.
\)
A implies 
$
o(\textbf{z})\sim o^{\pi}(\textbf{z}), \;\text{and}\; e(\textbf{z})\sim e^{\pi}(\textbf{z}),
$
and by WP we get 
$
e^{\pi}(\textbf{z})\prec o(\textbf{x}), \;\text{and}\; e(\textbf{x})\prec o^{\pi}(\textbf{z}).
$
By transitivity, it follows
$
e(\textbf{z})\prec o(\textbf{x}), \;\text{and}\; e(\textbf{x})\prec o(\textbf{z}),
$
which leads to 
$
z \in \Gamma.
$
\vspace{2mm}
\item Case $e(\textbf{x}) \sim o(\textbf{x})$: We remove $n_{m_j}$, $n_{m_j+1}$, for all $j>1$ from $U(x)$ to obtain $z \in 2^\w$ as follows:
\[
z(n) = \Big \{ 
\begin{array}{ll}
x(n) & \text{if $n\notin \left\{n_{m_j}, n_{m_j+1}: j>1\right\}$}\\
0 &  \text{otherwise}.
\end{array}
\]
By construction we obtain $o(\textbf{z})(n) \geq o(\textbf{x})(n)$ and $e(\textbf{z})(n) \leq e(\textbf{x})(n)$ for all $n\in \w$.
Further, for all $n > m_1$, $o(\textbf{z})(n) >  o(\textbf{x})(n)$ and $e(\textbf{z})(n) < e(\textbf{x})(n)$.
Applying a similar argument as in the proof of Claim \ref{C2}, by permuting finitely many elements, we are able to obtain $e^{\pi}(\textbf{z})$ and  $o^{\pi}(\textbf{z})$ such that 
\[
o(\textbf{x})(n) < o^{\pi}(\textbf{z})(n), \;\text{and}\; e^{\pi}(\textbf{z})(n) < e(\textbf{x})(n),\;\text{for all}\; n\in \w.
\]
Again, A implies 
$
o(\textbf{z})\sim o^{\pi}(\textbf{z}), \;\text{and}\; e(\textbf{z})\sim e^{\pi}(\textbf{z})
$,
WP implies 
$
o(\textbf{x})\prec o^\pi(\textbf{z}), \;\text{and}\; e^{\pi}(\textbf{z}) \prec e(\textbf{x}),
$
and therefore, by transitivity, it follows
\(
e(\textbf{z})\prec e(\textbf{x}), \;\text{and}\; o(\textbf{x})\prec o(\textbf{z}),
\)
which leads to 
$
z\in \Gamma.
$
\end{itemize}
\end{proof}

\section{Concluding remarks}

The aim of this paper was firstly motivated by answering Problem 11.14 in \cite{Mathias}, but we have then elaborated on more systematically the relationships between total SWRs and other irregular sets. These results might just be the tip of the iceberg of a potentially rather interesting research project, in order to use tools from infinitary combinatorics, forcing theory and descriptive set theory, to give a theoretical structure to the several social welfare relations on infinite utility streams defined in economic theory.
Other economic combinatorial principles which can be investigated are those \emph{\' a la Hammond}: given infinite utility streams $x,y \in X=Y^\omega$, we say that $x \leq_H y$ whenever there are $i \neq j$ such that $x(i) < y(i) < y(j) < x(j)$ and for all $k \neq i,j$, $x(k)=y(k)$.
Intuitively this type of pre-orders assert that a stream is better-off than another one if the distribution reduces the inequality among generations. 

So we consider to elaborate on the following idea: comparing different types of social welfare relations, in particular with respect to the following three categories: procedural equity principles (e.g. anonymity), efficiency principles (e.g. Pareto), consequentialist equity principles (e.g. Hammond), and describe a hierarchy of such relations based on the associated fragment of AC. From a pure theoretical point of view, this may suggest a ranking-method among combinations of the three kinds of principles, analysing a degree of compatibility between them.

This specifically means that one can expand the WR-diagram also with other statements combining these economic principles, or even introduce other similar WR-diagrams and then try to study the possible combinations of $\square$'s and $\blacksquare$'s. 

As a specific question left open in this paper, we consider the following being the most relevant: can one find a ZF-model satisfying $\IP \land \neg \SP$?

\section*{Appendix: on Shelah's amalgamation}

Let $\algebra{B}$ be a complete Boolean algebra and $\algebra{A} \lessdot \algebra{B}$. The \emph{projection} map $\pi: \algebra{B} \rightarrow \algebra{A}$ is defined by $\pi(b)=  \prod \{b \leq a: a \in \algebra{A} \}$.

Let $\algebra{B}$ be a complete Boolean algebra and $\algebraQ_1, \algebraQ_2$ two isomorphic complete subalgebras of $\algebra{B}$ and $\phi_0$ the isomorphism between them.
One defines the \enfa{amalgamation of $B$ over $\phi_0$}, say $\amal(\algebra{B}, \phi_0)$, as the complete Boolean algebra generated by the following set:
\(
\{ (b',b'') \in \algebra{B} \times \algebra{B}: \phi_0(\pi_1(b')) \cdot \pi_2(b'')\neq \mathbf{0} \},
\)
where $\pi_j: \algebra{B} \rightarrow \algebra{B}_j$ is the projection, for $j=1,2$.  Consider on $\amal(B, \phi_0)$ simply
the product order.
One can easily see that $e_j: \algebra{B} \rightarrow \amal(\algebra{B}, \phi_0)$ such that
\[
 e_1(b)= (b,\mathbf{1}) \text{ and } e_2(b)=(\mathbf{1},b)
\]
are both complete embeddings (\cite{JR93}, lemma 3.1), and for any $b_1 \in \algebra{B}_1$, one can show that
\( (b_1, \mathbf{1}) \text{ is equivalent to } (\mathbf{1}, \phi_0(b_1))\);
indeed, assume $(a',a'') \leq (b_1, \mathbf{1})$ and $(a',a'')$ incompatible with $(\mathbf{1}, \phi_0(b_1))$ (in $\amal(\algebra{B}, \phi_0)$). The former implies $\pi_1(a') \leq b_1$, while the latter implies $\pi_2(a'') \cdot \phi_0(b_1)=\mathbf{0}$, and hence one obtains $\phi_0(\pi_1(a'))\cdot \pi_2(a'')=\mathbf{0}$, which means that the pair $(a',a'')$ does not belong to the amalgamation.

Moreover, if one considers $f_1: e_2[\algebra{B}] \rightarrow e_1[\algebra{B}]$ such that, for every
$b \in \algebra{B}$, $f_1(\mathbf{1},b)=(b,\mathbf{1})$, one obtains an isomorphism between two copies of
$\algebra{B}$ into $\amal(\algebra{B}, \phi_0)$, such that $f_1$ is an extension of $\phi_0$ (since for every
$b_1 \in \algebra{B}_1$, $e_1(b_1)= (b_1, \mathbf{1})=(\mathbf{1}, \phi_0(b_1))=e_2(\phi_0(b_1))$, which means $e_1{\upharpoonright} B_1 = e_2\circ \phi_0$).

We can thus consider $e_1[\algebra{B}], e_2[\algebra{B}]$ as two isomorphic complete subalgebras of $\amal(\algebra{B}, \phi_0)$, and then repeat the same
procedure to construct
\[
\amal^2(\algebra{B}, \phi_0) := \amal(\amal(\algebra{B}, \phi_0), f_1)
\]
and $f_2$ the isomorphism between two copies of $\amal(\algebra{B}, \phi_0)$ extending
$f_1$. It is clear that one can continue such a construction, in order to define, for every $n \in \omega$,
\[
\amal^{n+1}(\algebra{B}, \phi_0) :=\amal(\amal^n(\algebra{B}, \phi_0), f_n)
\]
and $f_{n+1}$ the isomorphism between two copies of $\amal^n(\algebra{B}, \phi_0)$ extending $f_n$.
Finally, let
$\amal^\omega(\algebra{B}, \phi_0)$ be the Boolean completion of direct limit of $\amal^n(\algebra{B}, \phi_0)$'s, and
$\phi = \lim_{n \in \omega} f_n$.
One obtains $\algebra{B}_1, \algebra{B}_2 \lessdot \amal^\omega(\algebra{B}, \phi_0)$ and $\phi$ automorphism of
$\amal^\omega(\algebra{B}, \phi_0)$ extending $\phi_0$.
(N.B.: it is common in this framework to abuse terminology by referring to the \emph{Boolean completion} of the direct limit of a sequence of Boolean algebras simply as their \emph{direct limit}, and thus we write $\lim_{\alpha < \lambda} \algebra{B}_\alpha$ for the direct limit understood in this way.)

Note that the \emph{one-step amalgamation} $\amal(B,\phi_0)$ is forcing equivalent to a two step iteration $B_1 * (B/B_1 \times B/B_2)$, where remind $B_2:=\phi_0[B_1]$ and $B/B_1, B/B_2$ denote the quotient-algebras. More precisely we have 
\begin{lemma}  \label{lemma-amal-product}
Let $H$ be a $B_1$-generic filter over the ground model $V$. Then
\[
V[H] \models (B/B_1)^H \times (B/B_2)^{H} \text{ densely embeds into }(\amal(B, \phi_0)/e_1[B_1])^{H} 
\] 
\end{lemma}
For a proof one can see \cite[Lemma 3.2]{JR93}.

\end{document}